\documentclass[12pt]{article}
\usepackage[utf8]{inputenc}
\usepackage[english]{babel}
\usepackage{amsthm}
\usepackage{hyperref}
\usepackage{cleveref}
\usepackage{enumitem}
\usepackage{amsmath, amssymb, mathrsfs}
\usepackage{xypic}
\usepackage{graphicx}
\usepackage{aliascnt}
\newtheorem{theorem}{Theorem}[section]
\newtheorem{corollary}[theorem]{Corollary}
\newtheorem{lemma}[theorem]{Lemma}
\theoremstyle{definition}
\newtheorem{definition}[theorem]{Definition}

\theoremstyle{definition}
\newtheorem*{definition*}{Notational Conventions}

\theoremstyle{remark}
\newtheorem{remark}[theorem]{Remark}

\theoremstyle{definition}
\newtheorem{example}[theorem]{Example}

\title{Cocycles in Local Higher Category Theory}
\author{Nicholas J. Meadows}
\begin{document}

\maketitle

\section*{Introduction}

In \cite[Theorem 6.5]{local}, it is shown that given a right proper model category $M$, whose weak equivalences are closed under finite products, the maps $[X, Y]_{M}$ in the homotopy category of $M$ can be described as the path components of a $\textbf{cocycle category}$ $h(X, Y)_{M}$. Its objects are diagrams 
$$
X \xleftarrow{f} A \xrightarrow{g} Y
$$
 in $M$ with $f$ a weak equivalence and its morphisms are commutative diagrams 

$$
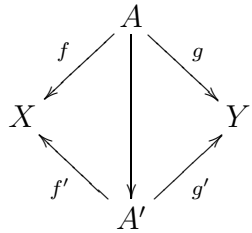
\xymatrix
{
& A \ar[dl]_{f} \ar[dd] \ar[dr]^{g} & \\
X & & Y \\
& A' \ar[ul]^{f'} \ar[ur]_{g'} &
}
$$

If $G$ is a sheaf of groups on a site $\mathscr{C}$, then a \textbf{$G$-torsor} is traditionally defined to be a sheaf $F$ with a principal and transitive $G$-action. This is equivalent to the fact that for the classical Borel construction $EG \times_{G} F$ the unique map
$$
EG \times_{G} F \rightarrow *
$$
is a local weak equivalence. 
Thus, every $G$-torsor determines a cocycle
$$
* \leftarrow EG \times_{G} F \rightarrow BG.
$$
This map induces a bijection
$$
\pi_{0}(\textbf{Tors}_{G}) \rightarrow \pi_{0} h(*, B(G))_{s\textbf{Pre}(\mathscr{C})}
$$
between path components of the category of $G$-torsors and path components of the cocycle category $h(*, B(G))_{s\textbf{Pre}(\mathscr{C})}$, leading to the homotopy classification of torsors (i.e. the homotopy theoretic interpretation of non-abelian $H^{1}$). \index{cohomology!non-abelian}

The technique of cocycles has numerous other applications, which are described in \cite[Chapter 9]{local} and \cite{Jardine-Cocycles}. These include homotopy classification of gerbes (non-abelian $H^{2}$) and an explicit model for stack completion.

The ultimate purpose of this paper is to show that cocycle-theoretic techniques apply to local higher category theory in the sense of \cite{Nick} and \cite{Nick2}. However, the problem is that neither the Joyal or the complete Segal model structures (and by extension their local analogues) are right proper. In fact, the only right proper model of higher categories is Bergner's model structure on simplicial categories. A major goal of this paper is to prove the existence of a local version of Bergner's model structure on simplicial categories and show that it is Quillen equivalent to the local Joyal model structure.

As an application of the theory of cocycles, we will prove a generalization of the homotopy classification of torsors. In particular, given an arbitrary presheaf of Kan complexes $X$, we will describe a bijection between the path components of a category of torsors and the maps
$$
[*, X]_{s\textbf{Pre}(\mathscr{C})}
$$
in the homotopy category of the injective model structure (\ref{cor5.13}).

The first two sections of the paper are devoted, respectively, to reviewing some basic aspects of local homotopy theory and model structures on simplicial categories (i.e. the Bergner model structure).   

In the third section, we define an appropriate local analogue of weak equivalences for the Bergner model structure. In the fourth section, we prove some auxiliary results related to Boolean localization and local fibrations. In the the fifth section, we will prove the existence of a local analogue of the Bergner model structure, and show that it is Quillen equivalent to the local Joyal model structure.

In the final section, the technique of cocycles is applied to describe the maps $[*, X]$ in the homotopy category of the local Bergner model structure, in the case that $\pi_{0}(X)$ is a presheaf of groupoids.

\section*{Basic Notational and Terminological Conventions}

For each $n \in \mathbb{N}$, write $[\textbf{n}]$  for the ordinal number category with $n+1$ objects $\{ 0, 1, \cdots, n-1, n \}$. We write $\mathrm{Set}, \mathrm{sSet}$ for the categories of sets and simplicial sets, respectively.

We will write $\mathrm{hom}_{C}(x, y)$ for the set of morphisms between two objects $x, y$ of a category $C$. Oftentimes, we will omit the subscript, because the category is obvious from the context. Given a small category $C$, we will write $\mathrm{Iso}(C)$ for the subcategory of $C$ whose objects are the objects of $C$ and whose morphisms are the isomorphisms of $C$. Given a small category $C$, we write $\mathrm{Mor}(C)$ and $\mathrm{Ob}(C)$ for the set of morphisms and objects of $C$, respectively. In the internal description of the category, the source, target, identity, and composition maps are respectively denoted $s, t, ident, c$.

Given a category $C$ enriched over simplicial sets, we will write $\mathrm{Ob}(C)$ and $\mathrm{Mor}(C)$ for the set of objects and the simplicial sets of morphisms, respectively. The source, target, identity, and composition maps are respectively denoted $s, t, ident, c$ in the internal description of the simplicially enriched category.

Given a small category $C$, write $B(C)$ for the nerve of $C$. Given two simplicial sets $K, Y$, write $X^{K}$ for the simplicial set whose $n$-simplices are maps $\mathrm{hom}(K \times \Delta^{n}, Y)$. The nerve functor has a left adjoint, which we denote 
$$
P : \mathrm{sSet} \rightarrow \mathrm{Cat}.
$$
Given a simplicial set $X$, we call $P(X)$ the \textbf{path category} of $X$. We write 
$$\pi(X) = P(X)[P(X)^{-1}].$$
This object is called the \textbf{fundamental groupoid} of $X$.

\section{Background on Local Homotopy Theory}

Throughout the paper, we will fix a small Grothendieck site $\mathscr{C}$. Thus, $s\textbf{Pre}(\mathscr{C})$ and $s\textbf{Sh}(\mathscr{C})$ are the simplicial presheaves and simplicial sheaves on $\mathscr{C}$, respectively. Throughout the paper, we fix a Boolean localization (see below) $p : s\textbf{Sh}(\mathscr{C}) \rightarrow s\textbf{Sh}(\mathscr{B})$.

Central to this paper is the use of the \textbf{injective model structure} on $s\textbf{Pre}(\mathscr{C})$, in which the weak equivalences are 'stalkwise weak equivalences' and cofibrations are monomorphisms (c.f. \cite[pg. 63-64]{local}). We will call its weak equivalences \textbf{local weak equivalences} and its fibrations \textbf{injective fibrations}. The proof of the existence of this model structure is found in \cite[Chapters 4 and 5]{local}. 

Recall that the \textbf{Joyal model structure} on simplicial sets is the unique model structure on $\mathrm{sSet}$ such that
\begin{enumerate}
\item{the cofibrations are monomorphisms.}
\item{the fibrant objects are quasi-categories.}
\end{enumerate}
The existence of the Joyal model structure is asserted in \cite[Theorem 2.2.5.1]{Lurie} and \cite[Theorem 6.12]{Joyal-quasi-cat}. We call its weak equivalences \textbf{Joyal equivalences} and its fibrations \textbf{quasi-fibrations}. A brief overview of some of its properties, useful for our purposes, is found in \cite{Nick}.

Recall from \cite[Theorem 3.3]{Nick} that there is a model structure, called the \textbf{local Joyal model structure}, on $s\textbf{Pre}(\mathscr{C})$ in which the weak equivalences are `stalkwise' Joyal equivalences. We call its weak equivalences \textbf{local Joyal equivalences} and its fibrations \textbf{quasi-injective fibrations}.

The technique of \textbf{Boolean localization} is used extensively thoughout the paper. A Boolean localization is a cover of a Grothendieck topos by the topos of sheaves on a complete Boolean algebra. The relevant properties of Boolean localization can be found in \cite[Section 2]{Nick}. However, a more comprehensive exposition is found in \cite[Chapter 3]{local}.

Suppose that $X, Y: \mathscr{C}^{op} \rightarrow C$ are functors, $f: X \rightarrow Y$ is a natural transformation and $i : K \rightarrow L$ is a map in $C$. We say that $f$ has \textbf{local right lifting property} with respect to $i$ if for every commutative diagram
$$
\xymatrix
{
 K \ar[r] \ar[d] & X(U) \ar[d] \\
L \ar[r] & Y(U)  
}
$$
with $U \in \mathrm{Ob}(\mathscr{C})$, there is exists a covering sieve $R \subseteq hom(-, U)$, such that the lift exists in the diagram
$$
\xymatrix
{
K \ar[r] \ar[d] & X(U) \ar[r]^{X(\phi)} & X(V) \ar[d]\\
L \ar[r] \ar@{.>}[urr] & Y(U) \ar[r]_{Y(\phi)} & Y(V)
}
$$
for each $\phi \in R$.

This leads to the various notion of local fibrations (e.g. local Kan fibrations), which are highly useful in local homotopy theory. The facts we need about local fibrations are summarized in \cite[Section 2]{Nick}. However, a more comprehensive exposition is found in \cite[Chapter 4]{local}. 

\section{The Dwyer-Kan and Bergner Model Structure}

We call a small category enriched in simplicial sets a \textbf{simplicial category}. We will denote the category of simplicial categories by $\mathrm{sCat}$.

Let $\mathcal{O}$ be a set. Write $\mathrm{sCat}_{\mathcal{O}}$ for the subcategory of $\mathrm{sCat}$ whose objects are simplicial categories $C$ with object set $\mathcal{O}$ and whose morphisms are the identity on objects.

In \cite{localization} it was proven that there was a proper model structure on  $\mathrm{sCat}_{\mathcal{O}}$ whose fibrations are maps which induce Kan fibrations of simplicial homs and whose weak equivalences are maps that induce weak equivalences of simplicial homs. We call this the \textbf{Dwyer-Kan model structure}.

Given a category $C$, we denote by $F(C)$ the free category with one generator for each non-identity map of $C$. There are maps $F^{2}(C) \rightarrow F(C)$ and $C \rightarrow F(C)$, which define a simplicial object $F_{*}(C)$ called the \textbf{simplicial resolution} of $C$ (c.f. \cite[Definition 2.5]{localization}).

Given a simplicial category $C$, we write $\mathrm{DK}(C)$ for the diagonal of the simplicial resolution of $C$. There is a natural weak equivalence $\mathrm{DK}(C) \rightarrow C$ which is in fact a cofibrant replacement, since the cofibrant objects in the Dwyer-Kan model structure are exactly the retracts of free objects (\cite[Proposition 7.6]{localization}).  
\\

Given a simplicial category $C$, one can construct a category $\pi_{0}(C)$, whose objects are the objects of $C$ and which satisfies $\mathrm{hom}_{\pi_{0}(C)}(x, y) = \pi_{0}\mathrm{hom}_{C}(x, y)$. A map $f \in \mathrm{hom}_{\mathcal{C}}(x, y)_{0}$ is called an \textbf{equivalence} if and only if it is an isomorphism in $\pi_{0}(C)$.

In \cite{BERGNER1}, Bergner constructs a right proper model category on the category of simplicial categories with the following properties.

\begin{enumerate}
\item{The weak equivalences (\textbf{sCat-equivalences}) are those maps $f: C \rightarrow D$ such that \begin{enumerate} \item{$\mathrm{hom}_{C}(x, y) \rightarrow \mathrm{hom}_{D}(f(x), f(y))$ are weak equivalences for all $ x, y \in C$.} 

\item{ $\pi_{0}(C) \rightarrow \pi_{0}(D)$ is essentially surjective.}
\end{enumerate}

}
\item{The fibrations (\textbf{sCat-fibrations}) are maps $f: C \rightarrow D$ such that
\begin{enumerate}
\item{$\mathrm{hom}_{C}(x, y) \rightarrow \mathrm{hom}_{D}(f(x), f(y))$ are Kan fibrations for all $ x, y \in C$.}
\item{Any equivalence $f(x) \rightarrow y$ in $D$ lifts to an equivalence $x \rightarrow z$ in $C$.}
\end{enumerate}
\item{The cofibrations are those maps which have the left lifting property with respect to maps which are both fibrations and weak equivalences.}
}
\end{enumerate}

The sCat-equivalences are referred to as DK-equivalences in \cite{BERGNER1}. 
\\

Condition b) in the definition of sCat-fibration can be replaced with the following condition 
\begin{enumerate}[label=(b')]
\item{$\pi_{0}(f)$ is an isofibration.}
\end{enumerate}

\begin{definition}\label{def1.1}

There is a functor $\mathcal{U} : \mathrm{sSet} \rightarrow \mathrm{sCat}$ such that $\mathcal{U}(S)$ is a simplicial category which has two objects x, y, $\mathrm{hom}_{\mathcal{U}(S)}(x, y) = S$ and $\mathcal{U}(S)$ has no other non-identity morphisms. 
\end{definition}

The Bergner model structure is cofibrantly generated with generating cofibrations
\begin{enumerate}
\item{$\mathcal{U}(\partial \Delta^{n}) \rightarrow \mathcal{U}(\Delta^{n})$ for $n \in \mathbb{N}$.}
\item{$\emptyset \rightarrow *$.}
\end{enumerate}

and generating trivial cofibrations
\begin{enumerate}
\item{The inclusions $\mathcal{U}(\Lambda_{i}^{n}) \rightarrow \mathcal{U}(\Delta^{n})$.}
\item{The inclusion maps $* \rightarrow \mathcal{H}$, where $\mathcal{H}$ runs over the set of isomorphism classes of simplicial categories with the following properties
\begin{enumerate}
\item{$\mathrm{Ob}(\mathcal{H}) = \{ x, y\}$.}
\item{The simplicial sets $\mathrm{hom}_{\mathcal{H}}(x, y), \mathrm{hom}_{\mathcal{H}}(x, x), \mathrm{hom}_{\mathcal{H}}(y, y)$ and $\mathrm{hom}_{\mathcal{H}}(y, x)$ are weakly contractible and have countably many non-degenerate simplices.}
\item{$* \xrightarrow{x} \mathcal{H}$ is a cofibration.}
\end{enumerate}
}
\end{enumerate}

\begin{lemma}\label{lem1.3}
The cofibrations in the Bergner model structure are monomorphisms. 
\end{lemma}

\begin{proof}
The non-trivial part is to show that a pushout of $i_{n} : \mathcal{U}(\partial \Delta^{n}) \rightarrow \mathcal{U}(\Delta^{n})$ is a monomorphism. Such a pushout is obtained in each simplicial degree by adding morphisms, so it suffices to show that the right hand map in the pushout
$$
\xymatrix
{
\{0, 1 \} \ar[d]_{(x, y)} \ar[r] & C \ar[d]_{g} \\
[\textbf{1}] \ar[r] & D
}
$$
is a monomorphism.

Consider the pushout
$$
\xymatrix
{
\partial \Delta^{1} \ar[d] \ar[r] & BC \ar[d]_{h} \\
\Delta^{1} \ar[r] & X
}
$$
Since $P$ preserves pushouts, $PBC \rightarrow PX$ is naturally isomorphic to $C \rightarrow D$. Note that $BC \rightarrow X$ is a monomorphism, so that $C \rightarrow D$ is a monomorphism on objects.
$X$ is obtained by adjoining a 1-simplex $\alpha$ to $BC$. 
Now, the morphism of $PX$ can be represented by strings of 1-simplices modulo an equivalence relation. Note that if $\alpha$ appears in a string 
$$
a_{0} \rightarrow \cdots a_{i} \rightarrow x \xrightarrow{\alpha} y \cdots a_{n},
$$
then no composition relation of $C$ can remove it. Thus, if $y_{1}, y_{2}$ are morphisms in $C$ and $y_{1} \simeq y_{2}$, then they must be equivalent by some composition laws in $C$, so that $y_{1} = y_{2}$.  

\end{proof}

\begin{definition}\label{def1.4}
Suppose that $\Omega^{*}$ is a cosimplicial object of a category $C$. Then there is a pair of adjoint functors associated to $\Omega^{*}$
$$
 |\, \, |_{\Omega^{*}} : sSet \leftrightarrows C : Sing_{\Omega^{*}}.   
$$
The left adjoint is given by  
$$
|S|_{\Omega^{*}} = \underset{\Delta^{n} \rightarrow S}{\underset{\longrightarrow}{lim}} \Omega^{n}
$$
and the right adjoint is given by $Sing_{\Omega^{*}}(S)_{n} = hom(\Omega^{n}, S)$. The right adjoint is known as a \textbf{singular functor associated to $\Omega^{*}$}.

\end{definition}

For each $n \in \mathbb{N}$ there is a simplicial category $\Phi^{n}$ such that
\begin{enumerate}
\item{
The objects $\Phi^{n}$ are the objects in the set $\{ 0, 1 \cdots, n\}$.}
\item{ $\mathrm{hom}_{\Phi^{n}}(i, j)$ can be identified with the nerve of the poset $\mathcal{P}_{n}[i, j]$ of subsets of the interval $[i, j]$ which contains the endpoints. That is, $\mathrm{hom}_{\Phi^{n}}(i, j) \cong (\Delta^{1})^{i-j-1}$.} 
\item{Composition is induced by union of posets.}
\end{enumerate}
These $\Phi^{n}$ glue together to give a cosimplicial object $\Phi$. The singular functor associated to $\Phi$ is called the \textbf{homotopy coherent nerve}, and is denoted $\mathfrak{B}$. We write $\mathfrak{C}$ for its left adjoint. 

The homotopy coherent nerve is significant because of the following theorem. 

\begin{theorem}\label{thm1.5}
There is a Quillen equivalence
$$
\mathfrak{C} : \mathrm{sSet} \leftrightarrows \mathrm{sCat} : \mathfrak{B}
$$
between the Joyal model structure and the Bergner model structure. 
\end{theorem}

This theorem appears as \cite[Theorem 2.2.5.1]{Lurie}, \cite[Corollary 8.2]{MappingSpaces} and \cite[Theorem 2.10]{JoyalSImplicial}. 
\\

We now want to construct a sectionwise fibration replacement of a morphism in the Bergner model structure. This will be important in proving the existence of the local Bergner model structure (c.f. \ref{lem4.5}). 

Given a simplicial category $C$ and $n \in \mathbb{N} $, we can construct a simplicial category $C^{(n)}$ such that its objects are objects of $C$ and $\mathrm{hom}_{C^{(n)}}(x, y) = \mathrm{hom}_{C}(x, y)^{\Delta^{n}}$. 

\begin{example}\label{exam1.6}

Let $f: X \rightarrow Y$ be a map of fibrant simplicial categories. Write $Y^{I}$ for the fibrant simplicial category $Y^{(1)}$. For $i = 0, 1$, let $d_{i} : Y^{I} \rightarrow Y$ be the map such that it is the identity on objects and for $x, y \in Y^{I}$, $\mathrm{hom}_{Y^{I}}(x, y) \rightarrow \mathrm{hom}_{Y}(x, y)$ is the map 
$$
\mathrm{hom}_{Y}(x, y)^{\Delta^{1}} \rightarrow \mathrm{hom}_{Y}(x, y)
$$
induced by $d^{i} : \Delta^{0} \rightarrow \Delta^{1}$. 
This map is a trivial fibration since $Y$ is fibrant. Thus, each $d_{i}$ is trivial sCat-fibration. Moreover, the $d_{i}$'s have a common section $s$, and one can apply a standard construction for categories of fibrant objects (c.f. the factorization lemma of \cite[pg. 421]{Brown}) to produce a functorial fibration replacement
\begin{equation*}
\xymatrix{
X \ar[r]^{s_{*}} \ar[dr]_{f} & Z_{f} \ar[d]_{\pi}  \\
& Y  
}
\end{equation*}
such that $s_{*}$ is the section of a trivial sCat-fibration.

\end{example}

Given a simplicial category $C$, write $\mathrm{Ex}^{\infty}(C)$ for the simplicial category obtained by applying $\mathrm{Ex}^{\infty}$ to the internal description of $C$ (i.e. $\mathrm{Mor}(C)$, $\mathrm{Ob}(C)$, composition operation, etc.). 

\begin{example}\label{exam1.7}

We can use the right properness of the Bergner model structure to construct a functorial fibration replacement for arbitrary maps of simplicial categories $f: X \rightarrow Y$. Consider the diagram
$$
\xymatrix
{
X \ar[rr]^{j} \ar[dd]_{f} \ar[dr]^{\theta_{f}} & & \mathrm{Ex}^{\infty}(X) \ar[dd]_>>>>>>>>>>>>>>>>>{\mathrm{Ex}^{\infty}(f)} \ar[dr]^{s_{*}}  & \\
& \tilde{Z}_{f} \ar[dl]^{\pi_{f}} \ar[rr]_>>>>>>>>>>>>>>>>>>>>>>>>>>{j_{*}} &  & Z_{\mathrm{Ex}^{\infty}(f)} \ar[dl]^{\pi} \\
Y \ar[rr]_{j} &  & \mathrm{Ex}^{\infty}(Y)  &  \\
}
$$
in which $\pi \circ s_{*}$ is the functorial factorization constructed in \ref{exam1.6} and the front face is a pullback. By the right properness of the Bergner model structure, $j_{*}$ is an sCat-equivalence. Thus, so is $\theta_{f}$. Finally, $\pi_{f}$ is an sCat-fibration, and the map $\pi_{f}$ is a fibration replacement of $f$.

This construction is functorial and commutes with filtered colimits. 
\end{example}

\section{Local sCat-Equivalences}

Let $f: C \rightarrow D$ be a functor. Consider the pullback of categories $\mathrm{Iso}(D)^{[\textbf{1}]} \times_{D} C$, whose objects are isomorphisms $f(c) \rightarrow d$ in $D$ and whose morphisms are commutative squares
$$
\xymatrix
{
f(c) \ar[r] \ar[d] & d \ar[d] \\
f(c') \ar[r] & d' 
}
$$
in $D$.
We can define a map $\phi_{f} : \mathrm{Iso}(D)^{[\textbf{1}]} \times_{D} C \rightarrow D$ by $(f(c) \rightarrow d) \mapsto d$. 
\begin{lemma}\label{lem2.1}
A functor $f : C \rightarrow D$ is an equivalence of categories if and only if 
\begin{enumerate}
\item{ $\mathrm{Mor}(C) \rightarrow \mathrm{Mor}(D) \times_{(\mathrm{Ob}(D) \times \mathrm{Ob}(D))} (\mathrm{Ob}(C) \times \mathrm{Ob}(C))$ is a bijection (fully faithful).}
\item{$\phi_{f} : \mathrm{Iso}(D)^{[\textbf{1}]} \times_{D} C \rightarrow D$
has the right lifting property with respect to $\emptyset \rightarrow *$ (essentially surjective).  }
\end{enumerate}
\end{lemma}

Throughout the rest of the document write $Cat\textbf{Pre}(\mathscr{C})$ and $sCat\textbf{Pre}(\mathscr{C})$ for the presheaves of categories on the site $\mathscr{C}$ and the presheaves of simplicial categories on $\mathscr{C}$, respectively. 

Given a map of presheaves of categories $f: C \rightarrow D$, we can form a pullback $\mathrm{Iso}(D)^{[\textbf{1}]} \times_{D} C$ such that $\mathrm{Ob}(\mathrm{Iso}(D)^{[\textbf{1}]} \times_{D} C)(U)$ consists of isomorphisms $f(c) \rightarrow d$ in $D(U)$. We also have a map 
$\phi_{f} : \mathrm{Iso}(D)^{[\textbf{1}]} \times_{D} C \rightarrow D$ which in each section is the map $\phi_{f}$ of \ref{lem2.1}. 

In light of the preceding lemma, we have the following definition. 

\begin{definition}\label{def2.2}

Suppose that $f: X \rightarrow Y$ is a map of presheaves of categories. Then we say that $f$ is a \textbf{local equivalence of presheaves of categories} if and only if
\begin{enumerate}
\item{The sheafification of the diagram 
$$
\xymatrix
{
\mathrm{Mor}(X) \ar[r] \ar[d] & \ar[d] \mathrm{Mor}(Y) \\
\mathrm{Ob}(X) \times \mathrm{Ob}(X) \ar[r] & \mathrm{Ob}(Y) \times \mathrm{Ob}(Y)
}
$$
is a pullback.
}
\item{
$\phi_{f} : \mathrm{Iso}(Y)^{[\textbf{1}]} \times_{Y} X \rightarrow Y$
has the local right lifting property with respect to $\emptyset \rightarrow *$.
}
\end{enumerate}
\end{definition}

\begin{definition}\label{def2.4}

We call a map $f: X \rightarrow Y$ of presheaves of simplicial categories a \textbf{local sCat-equivalence} if and only if 
\begin{enumerate}
\item{
The following diagram is homotopy cartesian for the injective model structure
$$
\xymatrix
{
\mathrm{Mor}(X) \ar[r] \ar[d] & \mathrm{Mor}(Y) \ar[d]_{(s, t)} \\
\mathrm{Ob}(X) \times \mathrm{Ob}(X) \ar[r] & \mathrm{Ob}(Y) \times \mathrm{Ob}(Y)
}
$$
}
\item{$\pi_{0}(X) \rightarrow \pi_{0}(Y)$ is a local equivalence of presheaves of categories. }
\end{enumerate}
\end{definition}

The following is \cite[Lemma 5.20]{local}.

\begin{lemma}\label{lem2.5}

Suppose we have a pullback diagram of simplicial presheaves
$$
\xymatrix{
B \times_{D} C \ar[r] \ar[d] & \ar[d] B \\
C  \ar[r]_{f} &  D
}
$$
where $f$ is a local Kan fibration. Then the diagram is homotopy cartesian for the injective model structure. 
\end{lemma}

\begin{remark}\label{rmk2.6}
\normalfont
Let $f: X \rightarrow Y$ be a map of simplicial presheaves. Note that $\mathrm{Ob}(X) \times \mathrm{Ob}(X) \rightarrow \mathrm{Ob}(Y) \times \mathrm{Ob}(Y)$ is a sectionwise Kan fibration. Thus, condition 1 of \ref{def2.4} holds if and only if $$\mathrm{Mor}(X) \rightarrow \mathrm{Mor}(Y) \times_{(\mathrm{Ob}(Y)\times \mathrm{Ob}(Y))} (\mathrm{Ob}(X) \times \mathrm{Ob}(X))$$ is a local weak equivalence.

In particular, this shows that condition 1 of \ref{def2.4} is a local analogue of condition (a) in the definition of sCat-equivalence. 
\end{remark}

\begin{remark}\label{rmk2.7}
\normalfont
It is clear from \ref{lem2.1} and \ref{rmk2.6} that a sectionwise  sCat-equivalence of presheaves of simplicial categories is also a local sCat-equivalence. 
\end{remark}

\begin{lemma}\label{lem2.8}
Suppose that $f: C \rightarrow D$ is a map of presheaves of simplicial categories. Then $f$ is a local sCat-equivalence if and only if 
\begin{enumerate}
\item{$f$ satisfies condition 1 of \ref{def2.4}.}
\item{The map $\phi_{\pi_{0}(f)} : (\mathrm{Iso}(\pi_{0}D))^{[\textbf{1}]} \times_{(\pi_{0}D)} (\pi_{0}C) \rightarrow (\pi_{0}D)$ of \ref{def2.2}
has the local right lifting property with respect to $\emptyset \rightarrow *$.}
\end{enumerate}
\end{lemma}

\begin{proof}
It suffices to show that these properties imply that $\pi_{0}(f)$ satisfies condition (1) of \ref{def2.2}. By \ref{rmk2.6}, we have a bijection
\begin{equation}\label{goldmine}
L^{2}\pi_{0}\mathrm{Mor}(X) \rightarrow L^{2} \pi_{0}(\mathrm{Mor}(Y) \times_{(\mathrm{Ob}(Y)\times \mathrm{Ob}(Y))} (\mathrm{Ob}(X) \times \mathrm{Ob}(X))).
\end{equation}
Since $\pi_{0}$ commutes with coproducts, there is a natural isomorphism $\pi_{0}\mathrm{Mor}(X) \cong \mathrm{Mor}(\pi_{0}(X))$, and by definition $\mathrm{Ob}(\pi_{0}X) = \mathrm{Ob}(X)_{0} = \pi_{0} \mathrm{Ob}(X)$. It follows that the map in \ref{goldmine}
is naturally isomorphic to the sheafification of
$$
\mathrm{Mor}(\pi_{0}X) \rightarrow \mathrm{Mor}(\pi_{0}Y) \times_{(\mathrm{Ob}(\pi_{0}Y)\times \mathrm{Ob}(\pi_{0}Y))} (\mathrm{Ob}(\pi_{0}X) \times \mathrm{Ob}(\pi_{0}X)),
$$
as required.  
\end{proof}

\begin{lemma}\label{lem2.9}
$X \rightarrow L^{2}(X)$ is a local sCat-equivalence.
\end{lemma}

\begin{proof}
The maps $\pi_{0}(X) \rightarrow \pi_{0}(L^{2}(X))$ and 
$$
\mathrm{Mor}(X) \rightarrow (\mathrm{Ob}(X) \times \mathrm{Ob}(X)) \times_{\mathrm{Ob}(L^{2}(X)) \times \mathrm{Ob}(L^{2}(X))} \mathrm{Mor}(L^{2}(X))
$$
both induce isomorphisms on associated sheaves. 
\end{proof}

\section{Boolean Localization and Local sCat-Fibrations}

Given a presheaf of simplicial categories $X$ and a simplicial category $C$, write $\mathrm{hom}(C, X)$ for the presheaf of simplicial categories defined by 
$$
U \mapsto \mathrm{hom}(C, X(U)).
$$

\begin{definition}\label{def3.1}
\normalfont
We call a map $f: X \rightarrow Y$ of presheaves of simplicial categories a \textbf{local trivial sCat-fibration} if and only if it has the local right lifting property with respect to all maps 
\begin{enumerate}
\item{$\emptyset \rightarrow *$.}
 \item{$\mathcal{U}(\partial \Delta^{n}) \rightarrow \mathcal{U}(\Delta^{n}), n \in \mathbb{N}$.}
\end{enumerate}
\end{definition}

We call a category $C$ \textbf{finite} if and only if it has a finite number of objects and each $\mathrm{hom}_{C}(x, y)$ is a finite set. We call a simplicial category $C$ \textbf{finite} if and only if it has a finite number of objects and each $\mathrm{hom}_{C}(x, y)$ is a finite simplicial set (i.e. has finitely many non-degenerate simplices).

\begin{lemma}\label{lem3.2}

There are isomorphisms 
$$
p^{*}L^{2}\mathrm{hom}(C, X) \rightarrow \mathrm{hom}(C, p^{*}L^{2}X) 
$$
natural in finite categories $C$ (respectively, finite simplicial categories $C$) and presheaves of categories $X$ (respectively presheaves of simplicial categories $X$).
\end{lemma}

\begin{proof}
 Suppose that $C$ is a category and $X$ is a presheaf of categories.
Note that $\mathrm{hom}(C, X)$ is naturally isomorphic to $\mathrm{hom}(BC, BX) \cong \mathrm{hom}(\mathrm{sk}_{2}(BC), BX)$. The simplicial set $\mathrm{sk}_{2}B(C)$ is finite. Thus, by the standard properties of Boolean localization (c.f. \cite[Lemma 2.6]{Nick}), we have 
$$
p^{*}L^{2}\mathrm{hom}(C, X) \cong \mathrm{hom}(\mathrm{sk}_{2}(BC), p^{*}L^{2}B(X)) \cong \mathrm{hom}(\mathrm{sk}_{2}BC, Bp^{*}L^{2}(X)) \cong \mathrm{hom}(C, p^{*}L^{2}X) .
$$ 

If $C$ is a finite simplicial category and $X$ is a presheaf of simplicial categories, form bisimplicial presheaves $C', X'$ such that $C'_{n,*} = \mathrm{sk}_{2}B(C_{n})$ and $X'_{n, *} = B(X_{n})$.  The constructions of $C'$ and $X'$ are natural, and there are natural isomorphisms
$$
\mathrm{hom}(C, X) \cong \mathrm{hom}(C', X').
$$  
Note that $C'$ is a finite bisimplicial set. 
Thus, we can use the basic properties of Boolean localization for bisimplicial presheaves (c.f. \cite[Lemma 1.17]{Nick2}) and the argument of the case of presheaves of categories to complete the proof.

\end{proof}

\begin{corollary}\label{cor3.3}
There is a natural isomorphism $p^{*}L^{2}\mathfrak{B} \cong \mathfrak{B}p^{*}L^{2}$. 
\end{corollary}

\begin{corollary}\label{cor3.4}
$p^{*}L^{2}$ preserves and reflects the property of having the local right lifting property with respect to a map of finite simplicial categories. In particular, $f$ is a local trivial sCat-fibration if and only if $p^{*}L^{2}(f)$ is a sectionwise trivial fibration in the Bergner model structure. 
\end{corollary}

If $X$ is a quasi-category, let $J(X)$ denote its maximal Kan subcomplex. If $X$ is a presheaf of quasi-categories, let $J(X)$ denote the functor $J$ applied sectionwise. 

\begin{lemma}\label{lem3.5}
If $C$ is a category, then $JB(C) \cong B\mathrm{Iso}(C)$. In particular, $\mathrm{Iso}$ applied sectionwise commutes with Boolean localization. 
\end{lemma}
\begin{proof}

For the first statement, note that by construction the $n$-simplices of $JB(C)$ are precisely the strings $a_{1} \rightarrow \cdots \rightarrow a_{n}$ of invertible arrows in $PB(C) \cong C$. 

For the second statement, note that both $B$ and $J$ commute with $p^{*}L^{2}$ (the latter by \cite[Lemma 3.6]{Nick}). Thus, for a presheaf of categories $C$ 
$$
Bp^{*}L^{2}\mathrm{Iso}(C) \cong p^{*}L^{2}B\mathrm{Iso}(C) \cong p^{*}L^{2}JB(C) \cong JBp^{*}L^{2}(C) \cong B\mathrm{Iso}(p^{*}L^{2}(C))
$$
so that $p^{*}L^{2}\mathrm{Iso}(C) \cong \mathrm{Iso}(p^{*}L^{2}(C))$.
\end{proof}

\begin{corollary}\label{cor3.6}
$p^{*}L^{2}$ preserves and reflects local equivalences of categories. 
\end{corollary}

\begin{proof}
The non-trivial part is showing that $p^{*}L^{2}$ preserves and reflects condition (2) of \ref{def2.2}. Given a map of presheaves of categories $f$, write $\phi_{f}$ for the map in condition (2) of \ref{def2.2}. We have $p^{*}L^{2}(\phi_{f}) \cong \phi_{p^{*}L^{2}(f)}$ by \ref{lem3.2} and \ref{lem3.5}. But $p^{*}L^{2}$ also preserves and reflects the property of being an epimorphism on objects, as required.  
\end{proof}

\begin{lemma}\label{lem3.7}
$p^{*}L^{2}$ preserves and reflects local sCat-equivalences. 
\end{lemma}

\begin{proof}
First note that $\pi_{0}$ is left adjoint to the functor $i: Cat\textbf{Pre}(\mathscr{C}) \rightarrow sCat\textbf{Pre}(\mathscr{C})$ which regards a presheaf of categories as a presheaf of discrete simplicial categories. We have natural isomorphisms $p_{*}i \cong ip_{*}$. Thus, by adjunction, we have a natural isomorphism
$$
p^{*}L^{2}\pi_{0} \cong L^{2}\pi_{0}p^{*}L^{2}.
$$

Now, condition (1) of \ref{lem2.8} is preserved and reflected under Boolean localization by \ref{rmk2.6}. For condition (2), note that if $f$ is a map of presheaves of simplicial categories, $\pi_{0}p^{*}L^{2}(f)$ is a local equivalence of categories if and only if  $ L^{2}\pi_{0}p^{*}L^{2}(f) \cong p^{*}L^{2}\pi_{0}(f)$ is local equivalence of categories. In turn, this is true if and only if if and only if $\pi_{0}(f)$ is a local equivalence of categories. 
\end{proof}

\begin{corollary}\label{cor3.8}
A local trivial sCat-fibration $f: X \rightarrow Y$ is a local sCat equivalence. 
\end{corollary}

\begin{lemma}\label{lem3.9}
Let $f: X \rightarrow Y$ be a map of presheaves of simplicial categories that is a sectionwise sCat-fibration.
Then $f$ is a local trivial sCat-fibration if and only if it is a local sCat-equivalence. 
\end{lemma}

\begin{proof}
Suppose that $f$ is a local sCat-equivalence.
Then $$\mathrm{Mor}(X) \rightarrow \mathrm{Mor}(Y) \times_{(\mathrm{Ob}(Y)\times \mathrm{Ob}(Y))} (\mathrm{Ob}(X) \times \mathrm{Ob}(X))$$ is a sectionwise Kan fibration and a local weak equivalence. By \cite[Theorem 4.32]{local}, it is a local trivial fibration. In particular, $f$ has the local right lifting property with respect to $\mathcal{U}(\partial \Delta^{n}) \rightarrow \mathcal{U}(\Delta^{n})$. 

 Choose an object $a \in Y(U)$. The map $$\phi_{f} : \mathrm{Ob}(\mathrm{Iso}(\pi_{0}(Y))^{[\textbf{1}]} \times_{(\pi_{0}Y)} (\pi_{0}X)) \rightarrow \mathrm{Ob}(\pi_{0}Y)$$ is a local epimorphism. Thus, there exists a covering $\{ U_{\alpha} \rightarrow U \}$ and equivalences $s_{\alpha} : f(b_{\alpha}) \rightarrow a|_{U_{\alpha}}$.

Since each $X(U_{\alpha}) \rightarrow Y(U_{\alpha})$ is an sCat-fibration, there exists an equivalence $s''_{\alpha} \in \mathrm{Mor}(X)(U_{\alpha})$ such that $f(s_{\alpha}'') = s_{\alpha}$. This means that $f(r_{\alpha}) = a|_{U_{\alpha}}$ for some $r_{\alpha}$. Thus, $f$ has the local right lifting property with respect to $\emptyset \rightarrow *$.

\end{proof}

\begin{lemma}\label{lem3.10}
Suppose that $f$ is a map of presheaves of fibrant simplicial categories. Then $f$ is a local sCat-equivalence if and only if $\mathfrak{B}(f)$ is a local Joyal equivalence. 
\end{lemma}
\begin{proof}
We can factorize $f$ as a sectionwise trivial cofibration for the Bergner model structure followed by a sectionwise sCat-fibration. Since $\mathfrak{B}$ is the right adjoint of a Quillen equivalence, it preserves and reflects weak equivalences between fibrant objects, so we can assume that $f$ is a sectionwise sCat-fibration. 

If $f$ is a local sCat-equivalence, it is a local trivial sCat-fibration by \ref{lem3.9}. But then $\mathfrak{B}p^{*}L^{2}(f) \cong p^{*}L^{2}\mathfrak{B}(f)$ is a sectionwise trivial fibration, so that $\mathfrak{B}(f)$ is a local Joyal equivalence. 

Conversely, suppose that $\mathfrak{B}(f)$ is a local Joyal equivalence. Then $p^{*}L^{2}\mathfrak{B}(f) \cong \mathfrak{B}p^{*}L^{2}(f)$ is a sectionwise quasi-fibration by  \cite[Lemma 3.14]{Nick} and \ref{thm1.5}. Thus, it is a sectionwise trivial fibration by \cite[Lemma 3.15]{Nick}. But $\mathfrak{B}$ reflects sectionwise weak equivalences between sectionwise fibrant objects, so that $p^{*}L^{2}(f)$ is a sectionwise sCat-equivalence. 
\end{proof}

\begin{corollary}\label{cor3.11}
Let $f: X \rightarrow Y$ be a map of sheaves of fibrant simplicial categories on a Boolean site $\mathscr{B}$. Then $f$ is a local sCat-equivalence if and only if it is a sectionwise sCat-equivalence.
\end{corollary}

\begin{proof}
Immediate from \cite[Corollary 3.10]{Nick}.
\end{proof}

\section{The Local Bergner Model Structure}

Since the Bergner model structure is cofibrantly generated, there is a \textbf{global projective model structure} on $sCat\textbf{Pre}(\mathscr{C})$ in which the fibrations and weak equivalences are respectively sectionwise fibrations and weak equivalences in the Bergner model structure. The cofibrations are called \textbf{projective cofibrations}. The generating set of projective cofibrations consists of objects of the form $\mathrm{hom}(-, U) \times \phi$, where $\phi$ is a generating cofibration for the Bergner model structure. 

The objective of this section is to prove the following theorem. 

\begin{theorem}\label{thm4.1}
There is a model structure on $sCat\textbf{Pre}(\mathscr{C})$ such that

\begin{enumerate}
\item{The cofibrations are projective cofibrations.}
\item{The weak equivalences are the local sCat-equivalences.}
\item{The fibrations are the maps which have the right lifting property with respect to maps which are both local sCat-equivalences and projective cofibrations. }
\end{enumerate} 
\end{theorem}

We call this the \textbf{Local Bergner Model Structure}. We call the fibrations \textbf{sCat-injective fibrations} and the injective objects \textbf{sCat-injective}.

\begin{lemma}\label{lem4.2}
Suppose that $F \in \textbf{Sh}(\mathscr{B})$ is a discrete sheaf. Then 
\begin{enumerate}
\item{$F$ is projective cofibrant.}
\item{For $n \in \mathbb{N}$, let $i_{n} : \mathcal{U}(\partial \Delta^{n}) \subseteq \mathcal{U}(\Delta^{n})$ be the inclusion. Then $F \times i_{n}$ is a projective cofibration.}
\end{enumerate} 
\end{lemma}

\begin{proof}
We prove the second statement; the first is similar. If $F$ is empty, the statement is trivial. Thus, we can assume $F \neq \emptyset$. 
Consider the poset $\mathcal{Y}$ of subsheaves $E$ of $F$ such that 
$$
E \times \mathcal{U}(\partial \Delta^{n}) \rightarrow E \times \mathcal{U}(\Delta^{n})
$$
is a projective cofibration. 

First note that this poset is nonempty. Choose an object $x \in E(b)$, and note that 
$E = \mathrm{hom}(-, b) \cong *|_{b} \xrightarrow{x} F$ is contained in $\mathcal{Y}$.    

 By way of contradiction, suppose that the largest element of this poset is $E \subsetneq F$. Then $E$ has a complement in $F$ since $\mathscr{B}$ is Boolean (\cite[Lemma 3.29]{local}). Write $E \coprod E' = F$. Choose an object $b \in \mathscr{B}$ such that $E'(b) \neq \emptyset$. Then  $(*|_{b} \coprod E) \times (i_{n}) \cong \mathrm{hom}(-, b) \coprod E \times (i_{n})$ is a projective cofibration, a contradiction. 

\end{proof}

\begin{lemma}\label{lem4.3}
If $f$ is a projective cofibration, then $p^{*}L^{2}(f)$ is isomorphic to $L^{2}(f')$ for some projective cofibration $f'$. 
\end{lemma}

\begin{proof}

If $g$ is a monomorphism of simplicial sets, then 
$$p^{*}L^{2}(\mathcal{U}(g) \times \mathrm{hom}(-, U)) \cong \mathcal{U}(g) \times F$$
 is a monomorphism for some sheaf $F$. It is therefore a projective cofibration by \ref{lem4.2}. 

By the first paragraph, $p^{*}L^{2}(f)$ is in the saturation (in the category $sCat\textbf{Sh}(\mathscr{B})$) of maps which have the right lifting property with respect to the sectionwise trivial sCat-fibrations. Thus, $p^{*}L^{2}(f)$ has the right lifting property with respect to all sectionwise trivial sCat-fibrations of sheaves of simplicial categories. 

One can factor $p^{*}L^{2}(f) = h \circ g$, where $g : X \rightarrow X'$ is a projective cofibration and $h : X' \rightarrow Y$ is a sectionwise trivial sCat-fibration.
A standard retract argument now shows that $f$ is the composite of a projective cofibration followed by sheafification.

\end{proof}

\begin{lemma}\label{lem4.4}
Suppose that we have a pushout diagram
$$
\xymatrix
{
A \ar[d]_{f} \ar[r] &  C \ar[d]^{g} \\
B \ar[r] & D
}
$$
where $f$ is a local sCat-equivalence and a projective cofibration. Then so is $g$.
\end{lemma}

\begin{proof}

 Let $\mathcal{Q}$ be the fibrant replacement functor for the global projective model structure on $sCat\textbf{Pre}(\mathscr{C})$. 
Consider the iterated pushout diagram
$$
\xymatrix
{
A \ar[r] \ar[d] & C \ar[r]_{j} \ar[d] & \mathcal{Q}(C) \ar[d]  \\
B \ar[r] &  D \ar[r] & \mathcal{Q}(C) \cup D
}
$$
The map $j$ is a sectionwise trivial cofibration for the Bergner model structure. Thus, we can assume that $C$ is a presheaf of fibrant simplicial categories. 
Form the iterated pushout diagram
$$
\xymatrix
{
A \ar[r] \ar[d] & \mathcal{Q}(A) \ar[d] \ar[r] & C \ar[d] \\
B \ar[r]_>>>>>>>{j'} & B \cup \mathcal{Q}(A) \ar[r] & B \cup C
}
$$
where the top horizontal composite is a factorization of $A \rightarrow C$. 
The map $j'$ is a sectionwise cofibration and weak equivalence for the Bergner model structure. Thus, we can assume that $A$ and $C$ are sectionwise fibrant. Using the argument of the first paragraph, we can assume $B$ is sectionwise fibrant as well.

If $A,B$ and $C$ are presheaves of fibrant simplicial categories, then 
$p^{*}L^{2}(f)$ is a sectionwise sCat-equivalence by \ref{cor3.11}. By using the argument at the end of the proof of \ref{lem4.3}, we can show that $p^{*}L^{2}(f)$ is the composite of a sectionwise trivial cofibration for the Bergner model structure $X \rightarrow X'$ and the natural map $X' \rightarrow L^{2}(X')$. One concludes that the pushout of $p^{*}L^{2}(f)$ is a local sCat-equivalence, as required.  

\end{proof}

\begin{lemma}\label{lem4.5}
Let $\alpha > |\mathrm{Mor}(\mathscr{C})|$ be a regular cardinal. 
Suppose that we have a diagram of monomorphisms
$$
\xymatrix
{
& X \ar[d] \\
A \ar[r] & Y
}
$$
where $A$ is $\alpha$-bounded and $X \rightarrow Y$ is a local sCat-equivalence. Then there exists an $\alpha$-bounded $B$, $A \subseteq B \subseteq Y$, such that $B \cap X \rightarrow B$ is a local sCat-equivalence. 
\end{lemma}

\begin{proof}
This is the same argument as \cite[Lemma 5.2]{local}, using \ref{exam1.7} and \ref{lem3.9}.
\end{proof}

Let $\beta > |\mathrm{Mor}(\mathscr{C})|$.
Then $\alpha = 2^{\beta} +1 $ is a regular cardinal. 
Let $\mathfrak{M}$ be collection of (representatives of isomorphism classes of) all $\alpha$-bounded maps which are local sCat-equivalences and monomorphisms. For each $m \in \mathfrak{M}$, form the factorization 
$$
\xymatrix{
C \ar[r]^{j_{m}} \ar[dr]_{m} & E  \ar[d]^{p_{m}} \\
& D
}
$$
where $j_{m}$ is a projective cofibration and $p_{m}$ has the right lifting property with respect to all projective cofibrations (by the same argument as in \cite[Theorem 4.8]{J1}, we can take $\alpha$ sufficiently large so that this factorization preserves $\alpha$-bounded objects). Let $\mathfrak{J}$ denote the set of all $j_{m}$ above.   Note that $j_{m}$ is a $\alpha$-bounded local sCat-equivalence. 
\\

The following lemma can be proven using the same argument as \cite[Lemma 7.3]{local}, along with \ref{lem4.5}. Note that \ref{lem4.5} applies in this case since every projective cofibration is a monomorphism by \ref{lem1.3}.

\begin{lemma}\label{lem4.6} 
Suppose that $q: X \rightarrow Y$ is a local sCat-equivalence which has the right lifting property with respect to all elements $j_{m}$ of the set $\mathfrak{J}$. Then $q$ has the right lifting property with respect to all projective cofibrations. 
\end{lemma}

Using \ref{lem4.6} and a standard retract argument, we obtain.

\begin{lemma}\label{lem4.7}
A map $f$ is an sCat-injective fibration if and only if it has the right lifting property with respect to all maps in the set $\mathfrak{J}$. 
\end{lemma}

\begin{proof}[Proof of Theorem 4.1]
CM5 follows from \ref{lem4.7} and \ref{lem4.4}. CM4 follows from a standard retract argument (c.f. the proof of \cite[Theorem 5.8]{local}). 
\end{proof}

Recall the main theorem of \cite{Joyal1}.

\begin{theorem}\label{thm4.10}
A quasi-category $X$ is a Kan complex if and only if $P(X)$ is a groupoid.  
\end{theorem}

\begin{lemma}\label{weirdlem}
Suppose that $X$ is a simplicial set. Then $\pi_{0}\mathfrak{C}(X) \cong P(X)$. 
\end{lemma}

\begin{proof}
We have isomorphisms $P(\Delta^{n}) \cong [\textbf{n}] \cong \pi_{0}\mathfrak{C}(\Delta^{n})$, natural in $n$. The result follows since a simplicial set is a colimit of its n-simplices and both $P$ and $\pi_{0}\mathfrak{C}$ are left adjoints. 
 
\end{proof}

\begin{lemma}\label{lem4.11}
$p^{*}L^{2}$ preserves sectionwise sCat-fibrations of presheaves of fibrant simplicial simplicial categories. 
\end{lemma}
\begin{proof}
Let $f : X \rightarrow Y$ be a sectionwise sCat-fibration. 
Condition (a) on pg. 7 is equivalent to $ \mathrm{Mor}(X) \rightarrow \mathrm{Mor}(Y) \times_{\mathrm{Ob}(Y) \times \mathrm{Ob}(Y)} (\mathrm{Ob}(X) \times \mathrm{Ob}(X))$ being a sectionwise Kan fibration, which is preserved under $p^{*}L^{2}$. 

On the other hand, $\mathfrak{B}p^{*}L^{2}(f) \cong p^{*}L^{2}\mathfrak{B}(f)$ is a sectionwise quasi-fibration by \cite[Lemma 3.15]{Nick}. Thus, $P\mathfrak{B}p^{*}L^{2}(f)$ is an isofibration in each section, since quasi-fibrations have the right lifting property with respect to $\Delta^{0} \rightarrow B(\pi \Delta^{1})$ (c.f. \cite[Corollary 1.6]{Joyal1}). We have natural equivalences $$P\mathfrak{B}p^{*}L^{2}(f) \cong \pi_{0}\mathfrak{C}\mathfrak{B}p^{*}L^{2}(f) \simeq \pi_{0}p^{*}L^{2}(f)$$ by \ref{thm1.5} and \ref{weirdlem}, so condition (b') on pg. 7 is verified. 
\end{proof}

We write $\mathrm{Ex}^{\infty} : sCat\textbf{Pre}(\mathscr{C}) \rightarrow sCat\textbf{Pre}(\mathscr{C})$ for the $\mathrm{Ex}^{\infty}$ functor defined above \ref{exam1.7} applied sectionwise. 

\begin{lemma}\label{lem4.12}
Consider a diagram
$$
\xymatrix{
C \times_{B} A \ar[r] \ar[d]_{h} & A \ar[d]_{g} \\
C \ar[r]_{f} & B
}
$$
where the $A \rightarrow B$ is a local sCat-equivalence and $C \rightarrow B$ is a sectionwise sCat-fibration. Then $C \times_{B} A \rightarrow C$ is an sCat-equivalence. 

In particular, the local Bergner model structure is right proper. 
\end{lemma}
\begin{proof}

$p^{*}L^{2}\mathrm{Ex}^{\infty}$ preserves pullbacks. $\mathrm{Ex}^{\infty}$ preserves sCat-fibrations, since it preserves $\pi_{0}$ and the usual  $\mathrm{Ex}^{\infty}$ for simplicial sets preserves Kan fibrations. Thus, $p^{*}L^{2}\mathrm{Ex}^{\infty}$ preserves sectionwise sCat-fibrations by \ref{lem4.11}. It also preserves and reflects local sCat-equivalences. Thus, we are reduced to assuming that all objects are sectionwise fibrant sheaves of simplicial categories on a Boolean site. The first statement follows from the right properness of the Bergner model structure. 

For the second statement, note that a fibration for the local Bergner model structure is a sectionwise sCat-fibration. 
\end{proof}

We call a map of simplicial presheaves a \textbf{projective cofibration} if and only if it is in the saturation of 
$$
\partial \Delta^{n} \times \mathrm{hom}(-, V) \rightarrow \Delta^{n} \times \mathrm{hom}(-, V),
$$
where $n$ runs over all natural numbers and $V$ runs over all objects of $\mathscr{C}$.

\begin{theorem}\label{thm4.13}
There is a model structure on $s\textbf{Pre}(\mathscr{C})$ in which
\begin{enumerate}
\item{Cofibrations are projective cofibrations.}
\item{Weak equivalences are local Joyal equivalences.}
\item{The fibrations are maps which have the right lifting property with respect to maps which are both local Joyal equivalences and projective cofibrations. We call these \textbf{projective quasi-fibrations}.}
\end{enumerate}
\end{theorem}

\begin{proof}

CM1-CM3 are trivial. Factorize a map $f = g \circ h$, where $g$ is an quasi-injective fibration (thus a projective quasi-fibration) and $h$ is a monomorphism and local Joyal equivalence. Then factor $h = l \circ m$, where $l$ has the right lifting property with respect to projective cofibrations (and is hence a projective quasi-fibration) and $m$ is a projective cofibration. The factorization $( g \circ l) \circ m$ gives one half of CM5. The other half is trivial. 

CM4 follows from a standard retract argument. 
\end{proof}

\begin{lemma}\label{lem4.14}
The identity map 
$$
i : s\textbf{Pre}(\mathscr{C}) \leftrightarrows s\textbf{Pre}(\mathscr{C}) : i
$$
is a Quillen equivalence from the local projective Joyal model structure to the (usual) local Joyal model structure. 
\end{lemma}
\begin{proof}
Trivial. 
\end{proof}

We write $\mathcal{L}_{Berg} : sCat\textbf{Pre}(\mathscr{C}) \rightarrow sCat\textbf{Pre}(\mathscr{C})$ for the functorial fibrant replacement for the local Bergner model structure. There is also a functor $\mathcal{S}_{Berg}$ which is obtained by applying the fibrant replacement functor for the Bergner model structure sectionwise to a presheaf of simplicial categories. 

\begin{theorem}\label{thm4.15}
There is a Quillen equivalence 
$$
\mathfrak{C} : s\textbf{Pre}(\mathscr{C}) \leftrightarrows sCat\textbf{Pre}(\mathscr{C}) : \mathfrak{B}
$$
from the local projective Joyal model structure to the local Bergner model structure. 
\end{theorem}

\begin{proof}
Clearly, $\mathfrak{C}$ takes generating cofibrations to cofibrations.  We want to show that $\mathfrak{C}$ sends local Joyal equivalences $f: X \rightarrow Y$ to local sCat-equivalences.
We have a diagram
$$
\xymatrix
{
X \ar[d]_{f} \ar[r] & \mathfrak{B} \mathcal{S}_{Berg} \mathfrak{C}(X) \ar[d]^{\mathfrak{B} \mathcal{S}_{Berg} \mathfrak{C}(f)} \\
Y \ar[r] & \mathfrak{B} \mathcal{S}_{Berg} \mathfrak{C}(Y)
}
$$
The horizontal maps are sectionwise Joyal equivalence by \ref{thm1.5}. By 2 out of 3, the right vertical map is a local Joyal equivalence. Thus, $\mathcal{S}_{Berg}\mathfrak{C}(f)$ and $\mathfrak{C}(f)$ are local sCat-equivalences by \ref{lem3.10}. 

Thus, the adjunction is a Quillen adjunction. 

Let $X$ be a fibrant object in the local Bergner model structure. Then $X$ is a presheaf of fibrant simplicial categories. Thus, $\mathfrak{C} \mathfrak{B}(X) \rightarrow X$ is a sectionwise sCat-equivalence by \ref{thm1.5}. 

Let $X$ be a simplicial presheaf. We have a commutative diagram
$$
\xymatrix{
X \ar[r]  \ar[d]_{id} & \mathfrak{B} \mathfrak{C}X \ar[r] \ar[d] & \mathfrak{B}\mathcal{S}_{Berg}\mathfrak{C}(X) \ar[d]_{\psi} \\
X \ar[r]_>>>>>{\gamma} & \mathfrak{B} \mathcal{L}_{Berg}\mathfrak{C}(X) \ar[r]_>>>>>{\phi} & \mathfrak{B}\mathcal{S}_{Berg}\mathcal{L}_{Berg}\mathfrak{C}(X) 
}
$$
The top horizontal composite is a sectionwise Joyal equivalence by \ref{thm1.5}. Moreover, by \ref{lem3.10}, $\phi, \psi$ are local Joyal equivalences. Thus, by two out of three, so is $\gamma$, as required. 
\end{proof}

\begin{remark}\label{rmk4.16}

The preceding two results show that there are Quillen equivalences relating all three models of local higher category theory.
\end{remark}

\section{The Homotopy Classification of Torsors}

\begin{definition}\label{def5.1}
\normalfont
Suppose that we have  model category $M$. For $X, Y \in \mathrm{Ob}(M)$ there is a category $h(X, Y)_{M}$ in which the objects are \textbf{cocycles}, i.e. diagrams 
$$
X \xleftarrow{f} A \xrightarrow{g} Y,
$$
where $f$ a weak equivalence. One writes $(f, g)$ for the cocycle depicted above. The morphisms in $h(X, Y)_{M}$ are commutative diagrams 
$$
\xymatrix
{
& A \ar[dl]_{f} \ar[dd] \ar[dr]^{g} & \\
X & & Y \\
& A' \ar[ul]^{f'} \ar[ur]_{g'} &
}
$$

\end{definition}

The following is \cite[Theorem 6.5]{local}.
\begin{theorem}\label{thm5.2}
Suppose that we have a model category $M$ such that
\begin{enumerate}
\item{Finite products preserve weak equivalences.}
\item{$M$ is right proper.}
\end{enumerate}
Then the natural map $\pi_{0}h(X, Y)_{M} \rightarrow [X, Y]_{M}$
defined by $$(X \xleftarrow{f} A \xrightarrow{g} Y) \mapsto g \circ f^{-1} $$
is a bijection. 
\end{theorem}

We write $\mathrm{sGpd}$ for the category of simplicial groupoids. 
\begin{theorem}\label{thm5.3} (\cite[Theorem V.7.7]{GJ2})
There is a model structure on $\mathrm{sGpd}$ in which the weak equivalences are maps $f$ such that they induce weak equivalences on simplicial homs and $\pi_{0}(f_{0})$ is a surjection. The fibrations are maps that induce Kan fibrations on simplicial homs and satisfy the path lifting property. 
\end{theorem}
For maps of simplicial groupoids, the path lifting property is equivalent to condition (b) in the definition of sCat-fibration, so that the fibrations coincide with the fibrations for the Bergner model structure. 

The weak equivalences are exactly the weak equivalences for the Bergner model structure. If $f : X \rightarrow Y$ is a map of simplicial groupoids $\pi_{0}(f)$ is essentially surjective if and only if $\pi_{0}\pi_{0}(f)$ is surjective. Now, if $G$ is a groupoid, then $a$ and $b$ lie in the same path component of $\pi_{0}(G)$ if and only if $\mathrm{hom}_{\pi_{0}G}(a, b) \neq \emptyset $. It follows that $\pi_{0}(G_{n}) \cong \pi_{0}(G_{0}) \cong \pi_{0}(\pi_{0}(G))$ for $n \ge 0$.
\\

We write $sGpd\textbf{Pre}(\mathscr{C})$ for the presheaves of simplicial groupoids on the site $\mathscr{C}$. We have the following local analogue of \ref{thm5.3}

\begin{theorem}\label{thm5.4}
(c.f \cite[Theorem 9.50, Lemma 9.52]{local}).
 There is a model structure on $sGpd\textbf{Pre}(\mathscr{C})$ defined as follows.
\begin{enumerate}
\item{The weak equivalences are maps $f: X \rightarrow Y$
such that 
\begin{enumerate}
\item{$f$ satisfies condition (1) of \ref{def2.4}.}
\item{$\pi_{0}(X_{0}) \rightarrow \pi_{0}(Y_{0}) $ is a local epimorphism.}
\end{enumerate}

}
\item{ A map $f$ is a fibration if and only if $\bar{W}(f)$ is an injective fibration.}
\item{The cofibrations are maps which have the left lifting property with respect to trivial fibrations.} 
\end{enumerate}

\end{theorem}

The weak equivalences in this model structure are exactly the local sCat-equivalences between presheaves of simplicial groupoids.

\begin{theorem}\label{thm5.5} (\cite[Theorem 9.50]{local})
There is a Quillen equivalence 
$$
\mathcal{G} : s\textbf{Pre}(\mathscr{C}) \leftrightarrows sGpd\textbf{Pre}(\mathscr{C}): \bar{W}
$$
between the injective model structure and the model structure of \ref{thm5.4}, where $\bar{W}$ is obtained by applying the Eilenberg-Maclane functor (c.f. \cite[V.7.7]{GJ2}) sectionwise.

\end{theorem}

The following is the main theorem of \cite{Stevenson} and is called the `generalized Eilenberg-Zilber Theorem' (c.f. also \cite[Proposition 9.38]{local}).

\begin{theorem}\label{thm5.6}
There is a natural weak equivalence $dB \rightarrow \bar{W}$. 
\end{theorem}

Suppose that $X$ and $Y$ are both presheaves of simplicial categories (respectively, presheaves of simplicial groupoids). Write $h_{hyp}(X, Y)_{s\mathrm{Cat}(\mathscr{C})}$ (respectively  $h(X, Y)_{s\mathrm{Gpd}(\mathscr{C})}$) for the full subcategory of $h(X, Y)_{s\mathrm{Cat}(\mathscr{C})}$ consisting of objects $(f, g)$ such that $f$ is also a sectionwise fibration for the Bergner model structure (respectively, for the model structure of \ref{thm5.3}).

\begin{lemma}\label{lem5.7}

Let $X$ and $Y$ are presheaves of simplicial categories such that $Y$ is a presheaf of fibrant simplicial categories. Let
$i: h_{hyp}(X, Y)_{sCat\textbf{Pre}(\mathscr{C})} \subseteq h(X, Y)_{sCat\textbf{Pre}(\mathscr{C})}$ be the inclusion. Then $\pi_{0}(i)$ is a bijection.

\end{lemma}

\begin{proof}
Use the argument of \cite[Lemma 6.14]{local}.
\end{proof}
If $X, Y$ are also presheaves of simplicial groupoids, then the statement is also true if we replace $sCat\textbf{Pre}(\mathscr{C})$ with $sGpd\textbf{Pre}(\mathscr{C})$.  

\begin{definition}\label{lem5.8}
\normalfont
We write $G : \mathrm{Cat} \rightarrow \mathrm{Gpd}$ for the groupoid completion functor. We write $G : \mathrm{sCat} \rightarrow \mathrm{sGpd}$ for the functor defined by $G(C)_{n} = G(C_{n})$. 
\end{definition}

\begin{lemma}\label{lem5.9}
Suppose that we have a local sCat-equivalence
$$
f: X \rightarrow Y,
$$
where $Y$ is a presheaf of groupoids, both $X$ and $Y$ are sectionwise fibrant and $X$ is projective cofibrant. Then the natural map 
$$
k_{X} : X \rightarrow G(X)
$$
is a local sCat-equivalence.
\end{lemma}

\begin{proof}

Note that $p^{*}L^{2}(f)$ is a sectionwise sCat-equivalence by \ref{cor3.11}. Thus, $\pi_{0}(p^{*}L^{2}X)$ is a presheaf of groupoids since $\pi_{0}(p^{*}L^{2}Y)$ is a presheaf of groupoids. Now, by \ref{lem4.3}, $p^{*}L^{2}(X) \cong L^{2}(X')$ for some projective cofibrant presheaf $X'$. Consider the diagram
$$
\xymatrix
{
GX'  & \ar[l]_{\phi} G \mathrm{DK} X' \ar[r]  & G \mathrm{DK} p^{*}L^{2} X \\
X'    \ar[u]_{k_{X'}} & \mathrm{DK} X'   \ar[u] \ar[l] \ar[r]  & \mathrm{DK}p^{*}L^{2}X  \ar[u]_{k_{\mathrm{DK}Gp^{*}L^{2}X}}
}
$$  
 The horizontal maps in the right hand square are both local isomorphisms. By \cite[Corollary 9.4]{localization}, the map $\phi$ is a sectionwise sCat-equivalence. Because $\pi_{0}p^{*}L^{2}(X)$ is a presheaf of groupoids, \cite[Proposition 9.5]{localization} implies $k_{\mathrm{DK}p^{*}L^{2}GX}$ is a sectionwise sCat-equivalence. We conclude that $X' \rightarrow G(X')$ is a local sCat-equivalence. But the sheafification of this map is naturally isomorphic to $p^{*}L^{2}(k_{X})$. Since $p^{*}L^{2}$ reflects local sCat-equivalences, we conclude that $k_{X}$ is a local sCat-equivalence. 

\end{proof}

\begin{theorem}\label{thm5.10}
Let $X$ and $Y$ be presheaves of simplicial groupoids. Then there is a bijection
$$
[X, Y]_{sGpd\textbf{Pre}(\mathscr{C})} \rightarrow [X, Y]_{sCat\textbf{Pre}(\mathscr{C})} 
$$
between maps in the homotopy category of the model structure of \ref{thm5.4}
and maps in the homotopy category of the local Bergner model structure.
\end{theorem}

\begin{proof}

We can replace $X$ and $Y$ with their fibrant replacement in the model structure of \ref{thm5.4}. In particular, we may assume that $X$ and $Y$ are presheaves of fibrant simplicial groupoids. 
Let
$$
i : h_{hyp}(X, Y)_{sGpd\textbf{Pre}(\mathscr{C})} \rightarrow h_{hyp}(X, Y)_{sCat\textbf{Pre}(\mathscr{C})}
$$
be the inclusion. By \ref{lem5.7} and \ref{thm5.2}, it suffices to show that $\pi_{0}(i)$ is a bijection.

First, note that $i$ is a surjection. Indeed, suppose that 
$$
\xymatrix
{
\sigma: X & \ar[l]_>>>>>>{f} Z \ar[r]^{g}  & Y 
}
$$
is an element of $ h_{hyp}(X, Y)_{sCat\textbf{Pre}(\mathscr{C})}$. Let $\mathcal{K}$ denote the cofibrant replacement functor for the global projective model structure on $sCat\textbf{Pre}(\mathscr{C})$. 
Since $Z$ is sectionwise fibrant,
it follows from \ref{lem5.9} that $G\mathcal{K}(Z) \rightarrow \mathcal{K}(Z)$ is a local sCat-equivalence. Thus,
$$
\xymatrix
{
& Z \ar[dl]_{f } \ar[dr]^{g } & \\
X  & & Y \\
& G(\mathcal{K}(Z)) \ar[ul] \ar[uu]_{h}  \ar[ur]&
}
$$
represents a map of cocycles, with the bottom cocycle an element of $h(X, Y)_{sGpd\textbf{Pre}(\mathscr{C})}$, as required.

On the other hand, we will show that $i$ is injective. It suffices to note that if  

$$
\xymatrix
{
& Z \ar[dl] \ar[dr] & \\
X  & & Y \\
& W \ar[ul] \ar[uu]  \ar[ur] &
}
$$
is a map of cocycles in which everything is sectionwise fibrant, then 
$$
\xymatrix
{
& G\mathcal{K}(Z) \ar[dl] \ar[dr] & \\
X  & & Y \\
& G\mathcal{K}(W) \ar[ul] \ar[uu]  \ar[ur] &
}
$$
represents a map of cocycles by \ref{lem5.9}.

\end{proof}

Given a presheaf of simplicial categories $A$, an \textbf{A-diagram} consists of a simplicial set map $\pi : X \rightarrow \mathrm{Ob}(A)$ as well as an action diagram
$$
\xymatrix
{
X \times_{s} \mathrm{Mor}(A) \ar[r]_>>>>>>>>{a} \ar[d] &  X \ar[d]_{\pi} \\
\mathrm{Mor}(A) \ar[r]_{t} &  \mathrm{Ob}(A)
}
$$
\index{$A$-diagram}
A map of $A$-diagrams is a commutative diagram
$$
\xymatrix
{
X \ar[r]_{\phi} \ar[dr]_{\pi} & Y \ar[d]^{\pi} \\
&  \mathrm{Ob}(A)
}
$$
which respects the action. 
We write $sSet^{A}$ for the category of $A$-diagrams. \index{$\,sSet^{A}$}

In sections, this is equivalent to the internal description of functors $X_{n} : A(U)_{n} \rightarrow \mathrm{Set}$, compatible with simplicial identities. 
Thus, we can define a simplicial presheaf $\underrightarrow{\mathrm{holim}}_{A_{n}}(X_{n})$
by 

 $$U \mapsto \underrightarrow{\mathrm{holim}}_{A_{n}(U)} X_{n}(U).$$ We call an $A$-diagram a \textbf{A-torsor} if and only if
$$
\underrightarrow{\mathrm{holim}}_{A_{n}}(X_{n}) \rightarrow *
$$
is a local weak equivalence. 
Let $\textbf{Tors}_{A}$\index{$\,\textbf{Tors}_{A}$} denote the full subcategory of $\mathrm{sSet}^{A}$ of $A$-torsors. This definition of torsors appeared in \cite{JardineTorsors}, and generalizes the classical description of a torsor (c.f. \cite[pg. 251]{local} for a discussion). \index{torsor}
\\

We call a presheaf $X$ of simplicial categories \textbf{sectionwise cofibrant} if and only if $X(U)$ is cofibrant for the Bergner model structure for all $ U \in \mathrm{Ob}(\mathscr{C})$. Note that this is not the same as being cofibrant for the local Bergner model structure.

\begin{theorem}\label{thm5.11} Suppose that $X$ is a sectionwise cofibrant presheaf of simplicial categories such that $\pi_{0}(X)$ is a presheaf of groupoids. Then there is a bijection
$$
\pi_{0}(\textbf{Tors}_{GX})  \rightarrow [*, X]_{sCat\textbf{Pre}(\mathscr{C})}.
$$

\end{theorem}

\begin{proof}
We have bijections 
$$
[*, X]_{sCat\textbf{Pre}(\mathscr{C})} = [*,GX]_{sCat\textbf{Pre}(\mathscr{C})} = [*, GX]_{sGpd\textbf{Pre}(\mathscr{C})} = [*, \bar{W}GX]_{s\textbf{Pre}(\mathscr{C})}= [*, dBGX]_{s\textbf{Pre}(\mathscr{C})},
$$
where $[\, , \, ]_{s\textbf{Pre}(\mathscr{C})}$ denotes homotopy classes of maps in the injective model structure. The first bijection follows from \cite[Proposition 9.5]{localization} and the fact that $\pi_{0}X$ is a presheaf of groupoids. The second and third follow from \ref{thm5.10} and \ref{thm5.5}, respectively. The final one comes from \ref{thm5.6}.

On the other hand, \cite[Theorem 24]{JardineTorsors} gives a bijection
$$
[*, dB G X]_{s\textbf{Pre}(\mathscr{C})} = \pi_{0}(\textbf{Tors}_{G X}) .
$$
\end{proof}

\begin{corollary}\label{cor5.12}
Suppose that $X$ is a presheaf of simplicial categories such that $\pi_{0}(X)$ is a presheaf of groupoids. Then there is a bijection
$$
\pi_{0}(\textbf{Tors}_{G\mathrm{DK}(X)})  \rightarrow [*, X]_{sCat\textbf{Pre}(\mathscr{C})}.
$$
\end{corollary}

\begin{corollary}\label{cor5.13}
Suppose that $X$ is a presheaf of Kan complexes. Then we have a bijection
$$
\pi_{0}(\textbf{Tors}_{G  \mathfrak{C}(X)})  \rightarrow [*, X]_{s\textbf{Pre}(\mathscr{C})},
$$
where $[\, , \, ]_{s\textbf{Pre}(\mathscr{C})} $ denotes homotopy classes of maps in the injective model structure.
\end{corollary}

\begin{proof}

First, note that by \cite[Theorem 4.10]{Nick2} and the fact that $X$ is a presheaf of Kan complexes, we have a bijection
$$
[*, X]_{inj} = [*, X]_{LJoyal}
$$
where $[*, X]_{inj}$ and $[\, , \, ]_{LJoyal}$ denotes maps in the homotopy categories of the injective and local Joyal model structures, respectively. The Quillen equivalences of \ref{lem4.14} and \ref{thm4.15} imply that there are bijections 
$$
[*, X]_{LJoyal}  = [*, \mathfrak{C}(X)]_{LBerg}
$$
between maps in the homotopy categories of the local Joyal and local Bergner model structures, since Quillen equivalences induce equivalences of homotopy categories. 

Since everything in the Joyal model structure is cofibrant, $\mathfrak{C}(X)$ is sectionwise cofibrant presheaf by the Quillen equivalence of \ref{thm1.5}. By \ref{weirdlem}, $\pi_{0}\mathfrak{C}(X) \cong P(X)$. But $P(X)$ is a groupoid by \ref{thm4.10}. Thus, $\mathfrak{C}(X)$ satisfies the hypotheses of \ref{thm5.11} and we have an identification
$$
[*, \mathfrak{C}(X)]_{sCat\textbf{Pre}(\mathscr{C})} = \pi_{0}(\textbf{Tors}_{G\mathfrak{C}X}),
$$
from which the result follows. 
\end{proof}

\begin{remark}\label{rmk5.14}
\normalfont
In \cite[Theorem 2.3]{Riehl1} an explicit description of $\mathfrak{C}(X)$ for a quasi-category $X$ is given that may prove particularly useful for calculations. From this description, a number of interesting properties of $\mathfrak{C}(X)$ are deduced, such as the fact that its simplicial homs are 3-coskeletal (\cite[Theorem 4.1]{Riehl1}).  
\end{remark}

\bibliography{cocyclesbib.bib}

\end{document}